\documentclass[11pt]{article}

\usepackage{amsmath,amssymb,amsthm,mathtools}
\usepackage[T1]{fontenc}

\usepackage{tikz-cd}

\usepackage{authblk}

\title{\bf Haar-Type Measures on Topological Quasigroups and Kunen's Theorem}

\author{\Large Takao Inou\'{e}}

\affil{\large Faculty of Informatics, Yamato University, \\ Osaka, Japan\footnote{Email: inoue.takao@yamato-u.ac.jp; \\ Personal Email: takaoapple@gmail.com \\ [I prefer my personal email address for correspondence.]\\ This is the second version.}} 
\date{March 11, 2026}

\newtheorem{definition}{Definition}
\newtheorem{remark}{Remark}
\newtheorem{proposition}{Proposition}
\newtheorem{theorem}{Theorem}
\newtheorem{conjecture}{Conjecture}
\newtheorem{lemma}{Lemma}

\newtheorem{problem}{Problem}

\begin{document}

\maketitle

\begin{abstract}
Haar measure is a fundamental structure in harmonic analysis on locally compact
groups.  Its existence reflects the compatibility between topology and the
associative algebraic structure of groups.  In this paper we propose a
framework for Haar-type measures on topological quasigroups.  Since
associativity is absent, strict translation invariance is generally too strong
to expect.  We therefore introduce quasi-invariant measures whose defect is
measured by a modular cocycle attached to translations.

We then explain, in a detailed and cautious form, how Moufang-type identities
may impose strong constraints on this cocycle.  In particular, under additional
quasi-invariance assumptions for right translations, the Moufang-type identity
$(N1)$ leads naturally to a multiplicativity relation for the cocycle.  This
suggests a measure-theoretic interpretation of Kunen's theorem: the emergence
of loop structure may be viewed as the collapse of a modular defect in the
translation geometry of a quasigroup.

This perspective suggests a possible measure-theoretic
interpretation of Kunen's theorem in terms of unimodularity
of the translation geometry.
\end{abstract}

\medskip


\tableofcontents

\section{Introduction}

Haar measure is one of the central structures in the theory of locally compact
groups.  If $G$ is a locally compact group, then there exists a nonzero regular
Borel measure $\mu$ such that
\[
\mu(gE)=\mu(E)
\qquad (g\in G),
\]
for every Borel set $E\subseteq G$.  Equivalently, if $L_g:G\to G$ denotes left
translation,
\[
L_g(x)=gx,
\]
then
\[
(L_g)_*\mu=\mu
\qquad (g\in G).
\]
This compatibility between the topology and the translation structure is one of
the foundations of abstract harmonic analysis
\cite{Haar1933,HewittRoss,Folland}.

A natural question is whether an analogous structure can exist for
\emph{topological quasigroups}.  Quasigroups retain the bijectivity of left and
right translations, but they lack associativity
\cite{Pflugfelder,Smith,Bruck}.  Consequently, the classical proof of the Haar
theorem cannot be transferred directly.  In particular, the family of left
translations is no longer modeled by the quasigroup itself in the way it is for
groups.

The purpose of this paper is not to claim a full Haar theorem for arbitrary
topological quasigroups.  Rather, the aim is to formulate a careful
measure-theoretic framework in which the failure of strict translation
invariance is encoded by a positive cocycle.  This cocycle plays the role of a
modular defect.  We then explain how Moufang-type identities may constrain this
defect.

The central motivating example is Kunen's theorem \cite{Kunen1996}, which
states that a quasigroup satisfying the Moufang-type identity $(N1)$ is
necessarily a loop.  The identity is
\[
((xy)z)y = x(y(zy)).
\]
Our guiding idea is that, under suitable quasi-invariance hypotheses, $(N1)$
forces strong compatibility relations among translation operators, and these in
turn constrain the modular cocycle.  In this sense, loop structure may be
viewed as a kind of unimodularity phenomenon in the translation geometry of a
quasigroup.

The purpose of this paper is to propose a framework in which Haar-type
structures can be studied in the quasigroup setting using quasi-invariant
measures and modular cocycles.

This paper proposes a measure-theoretic viewpoint on quasigroup structure,
suggesting that loop identities may be interpreted as a form of
unimodularity in the translation geometry of topological quasigroups.

\begin{remark}
The present paper is primarily conceptual in nature.
We do not claim a general existence theorem for Haar-type
measures on arbitrary topological quasigroups.
Instead, the goal is to isolate a structural mechanism
linking translation operators, modular cocycles,
and Moufang-type identities.
\end{remark}

\section{Topological Quasigroups and Translation Operators}

\begin{definition}
A \emph{quasigroup} is a set $Q$ equipped with a binary operation
\[
Q\times Q\to Q,\qquad (x,y)\mapsto xy,
\]
such that for every $a,b\in Q$ there exist unique elements $x,y\in Q$ with
\[
ax=b,\qquad ya=b.
\]
Equivalently, the equations $au=b$ and $va=b$ are uniquely solvable in $u$ and
$v$.
\end{definition}

\begin{definition}
A \emph{topological quasigroup} is a quasigroup $Q$ endowed with a Hausdorff
topology such that multiplication and both division operations are continuous.
\end{definition}

For each $a\in Q$ we define the left and right translations
\[
L_a(x)=ax,\qquad R_a(x)=xa.
\]
By the quasigroup axioms, each $L_a$ and each $R_a$ is a bijection.  In a
topological quasigroup they are in fact homeomorphisms, since the inverse maps
are given by the corresponding division operations.

\begin{remark}
For groups one has
\[
L_a\circ L_b = L_{ab}.
\]
This identity is one of the key reasons why Haar measure is available in the
group setting.  In a quasigroup this relation generally fails, because
\[
(L_a\circ L_b)(x)=a(bx)
\]
need not be equal to
\[
(ab)x=L_{ab}(x).
\]
Thus the loss of associativity appears immediately at the level of translation
operators.
\end{remark}

Nevertheless, left and right translations still generate genuine groups of
homeomorphisms.

\begin{definition}
The \emph{left multiplication group} of a quasigroup $Q$ is the subgroup
\[
\mathrm{LMlt}(Q)\leq \mathrm{Homeo}(Q)
\]
generated by all left translations $L_a$ $(a\in Q)$.
Similarly, the \emph{right multiplication group} $\mathrm{RMlt}(Q)$ is the
subgroup generated by all right translations, and the
\emph{multiplication group} $\mathrm{Mlt}(Q)$ is generated by both left and
right translations.
\end{definition}

\section{The Moufang Identity and Translation Calculus}

We now turn to the Moufang-type identity
\[
((xy)z)y = x(y(zy)).
\tag{N1}
\]

The key point is that this identity can be rewritten as an operator identity on
$Q$.

\begin{theorem}[Translation form of the Moufang identity (N1)]
\label{thm:translation-form-N1}
Let $Q$ be a quasigroup.
Then the Moufang identity
\[
((xy)z)y = x(y(zy))
\]
holds for all $x,y,z\in Q$ if and only if the translation operators satisfy
\[
R_y \circ L_{xy} = L_x \circ L_y \circ R_y
\qquad (x,y\in Q).
\]
\end{theorem}

\begin{proof}

Assume first that the Moufang identity
\[
((xy)z)y = x(y(zy))
\]
holds for all $x,y,z\in Q$.

Let $x,y\in Q$ be fixed and let $z\in Q$ be arbitrary.
We compute the action of the operators on $z$.

First,
\[
(R_y \circ L_{xy})(z)
=
R_y((xy)z)
=
((xy)z)y .
\]

On the other hand,

\[
(L_x \circ L_y \circ R_y)(z)
=
L_x(L_y(R_y(z))).
\]

Since
\[
R_y(z)=zy,
\]
we obtain

\[
L_y(R_y(z)) = y(zy).
\]

Applying $L_x$ gives

\[
L_x(y(zy)) = x(y(zy)).
\]

Therefore

\[
(L_x \circ L_y \circ R_y)(z)
=
x(y(zy)).
\]

By the Moufang identity these two expressions coincide, hence

\[
(R_y \circ L_{xy})(z)
=
(L_x \circ L_y \circ R_y)(z)
\]

for every $z\in Q$.

Thus

\[
R_y \circ L_{xy} = L_x \circ L_y \circ R_y .
\]

Conversely, suppose the translation identity

\[
R_y \circ L_{xy} = L_x \circ L_y \circ R_y
\]

holds for all $x,y\in Q$.
Applying both sides to an arbitrary element $z\in Q$ gives

\[
(R_y \circ L_{xy})(z)
=
(L_x \circ L_y \circ R_y)(z).
\]

Expanding both sides yields

\[
((xy)z)y = x(y(zy)).
\]

Thus the Moufang identity holds.

\end{proof}

\section{Haar-Type Measures and Modular Cocycles}

\subsection{Quasi-invariance under left translations}

The natural analogue of Haar measure in the quasigroup setting should interact
with the translation structure.  Since strict invariance is too strong to
expect in general, we begin with quasi-invariance.
Let $T:X\to X$ be a measurable map and let $\mu$ be a measure on $X$.
The push-forward measure $T_*\mu$ is defined by
\[
(T_*\mu)(E) = \mu(T^{-1}(E))
\]
for every measurable set $E$.

Equivalently, for every nonnegative measurable function $f$,
\[
\int f(x)\,d(T_*\mu)(x)
=
\int f(T(x))\,d\mu(x).
\]

\begin{definition}
Let $Q$ be a locally compact Hausdorff topological quasigroup.  A regular
Borel measure $\mu$ on $Q$ is called a \emph{left quasi-invariant measure}
if for every $a\in Q$ there exists a positive scalar $j(a)$ such that
\[
(L_a)_*\mu = j(a)\mu.
\]
The function
\[
j:Q\to \mathbb{R}_{>0}
\]
is called the \emph{left modular cocycle} associated with $\mu$.
\end{definition}

\begin{remark}
The existence of such quasi-invariant measures is a nontrivial
problem in general.
Unlike the group case, there is currently no general analogue
of the Haar existence theorem for topological quasigroups.
In this paper we work conditionally under the assumption
that such measures are given.
\end{remark}

Here $(L_a)_*\mu$ is the push-forward measure defined by
\[
((L_a)_*\mu)(E)=\mu(L_a^{-1}(E))
\]
for Borel sets $E\subseteq Q$.

The same definition can be expressed at the level of integration.  Namely, for
every nonnegative Borel function $f$ one has
\[
\int_Q f(x)\,d((L_a)_*\mu)(x)
=
\int_Q f(L_a(x))\,d\mu(x).
\]
Hence the condition
\[
(L_a)_*\mu=j(a)\mu
\]
is equivalent to
\[
\int_Q f(ax)\,d\mu(x)
=
j(a)\int_Q f(x)\,d\mu(x)
\]
for all nonnegative Borel $f$ (and therefore also for all integrable $f$).

\begin{remark}
In the group case, if one works with a left Haar measure, then left
translations preserve the measure exactly, so the corresponding left cocycle is
identically equal to $1$.  The nontrivial modular function arises from right
translations.  In the quasigroup setting we do not have that asymmetry a priori,
so it is natural to allow a cocycle already on the left.
\end{remark}

\subsection{Two-sided quasi-invariance}

The basic calculation involving the Moufang identity naturally mixes left and
right translations.  For that reason, if one wants to perform actual
measure-theoretic derivations, it is convenient to impose quasi-invariance for
right translations as well.

\begin{definition}
A regular Borel measure $\mu$ on a locally compact Hausdorff topological
quasigroup $Q$ is called \emph{two-sided quasi-invariant} if there exist
positive functions
\[
j,\rho:Q\to\mathbb{R}_{>0}
\]
such that
\[
(L_a)_*\mu = j(a)\mu,\qquad (R_a)_*\mu=\rho(a)\mu
\]
for all $a\in Q$.
\end{definition}

Again, the right quasi-invariance condition is equivalent to
\[
\int_Q f(xa)\,d\mu(x)
=
\rho(a)\int_Q f(x)\,d\mu(x)
\]
for all suitable $f$.

\begin{remark}
The assumption of two-sided quasi-invariance is admittedly
strong.  It is introduced here in order to make the cocycle
calculation transparent when both left and right translations
appear in the Moufang identity.  The existence and naturality
of such measures remain open problems.
\end{remark}

\begin{remark}
This two-sided hypothesis is strong, and we do not claim here that it holds for
all locally compact topological quasigroups.  Its role is methodological: it
provides a clean framework in which one can derive cocycle relations from
translation identities.  The existence problem for such measures is discussed
later as part of the research program rather than as an established theorem.
\end{remark}

\subsection{Elementary consequences}

We record a simple but useful functorial fact about push-forwards.

\begin{lemma}
Let $X$ be a locally compact Hausdorff space, let $\mu$ be a regular Borel
measure on $X$, and let $S,T:X\to X$ be homeomorphisms.  Then
\[
(S\circ T)_* \mu = S_*(T_*\mu).
\]
\end{lemma}

\begin{proof}
For every Borel set $E\subseteq X$,
\[
((S\circ T)_*\mu)(E)
=
\mu((S\circ T)^{-1}(E))
=
\mu(T^{-1}(S^{-1}(E)))
=
(T_*\mu)(S^{-1}(E))
=
(S_*(T_*\mu))(E).
\]
\end{proof}

Applying this to left and right translations gives the basic computational rule
used throughout the paper.

\section{The Moufang Identity and Translation Calculus}

We now turn to the Moufang-type identity
\[
((xy)z)y = x(y(zy)).
\tag{N1}
\]

The key point is that this identity can be rewritten as an operator identity on
$Q$.

\section{Deriving Cocycle Relations from $(N1)$}

We now explain in detail how $(N1)$ constrains the modular cocycles, under the
additional assumption of two-sided quasi-invariance.

The Moufang identity can be expressed naturally in terms
of translation operators.  This formulation will be useful
for the measure-theoretic computations below.

\begin{proposition}
Let $Q$ be a locally compact Hausdorff topological quasigroup satisfying
$(N1)$.  Assume that $\mu$ is a two-sided quasi-invariant measure on $Q$, so
that
\[
(L_a)_*\mu=j(a)\mu,\qquad (R_a)_*\mu=\rho(a)\mu
\]
for positive functions $j,\rho:Q\to\mathbb{R}_{>0}$.
Then for all $x,y\in Q$,
\[
j(xy)\rho(y)=j(x)j(y)\rho(y).
\]
Since $\rho(y)>0$, this simplifies to
\[
j(xy)=j(x)j(y).
\]
\end{proposition}

\begin{proof}
Fix $x,y\in Q$.  By the operator form of $(N1)$ established above, we have
\[
R_y\circ L_{xy}=L_x\circ L_y\circ R_y.
\]
Taking push-forwards and using the functoriality lemma,
\[
(R_y\circ L_{xy})_*\mu
=
(R_y)_*((L_{xy})_*\mu),
\]
and similarly
\[
(L_x\circ L_y\circ R_y)_*\mu
=
(L_x)_*((L_y)_*((R_y)_*\mu)).
\]

We now compute each side separately.

For the left-hand side,
\[
(L_{xy})_*\mu=j(xy)\mu
\]
by definition of $j$.  Applying $(R_y)_*$ gives
\[
(R_y)_*((L_{xy})_*\mu)
=
(R_y)_*(j(xy)\mu)
=
j(xy)(R_y)_*\mu
=
j(xy)\rho(y)\mu.
\]

For the right-hand side, first use
\[
(R_y)_*\mu=\rho(y)\mu.
\]
Then apply $(L_y)_*$:
\[
(L_y)_*((R_y)_*\mu)
=
(L_y)_*(\rho(y)\mu)
=
\rho(y)(L_y)_*\mu
=
\rho(y)j(y)\mu.
\]
Now apply $(L_x)_*$:
\[
(L_x)_*((L_y)_*((R_y)_*\mu))
=
(L_x)_*(\rho(y)j(y)\mu)
=
\rho(y)j(y)(L_x)_*\mu
=
\rho(y)j(y)j(x)\mu.
\]

Since the two composite transformations are equal, their push-forwards of $\mu$
are equal.  Hence
\[
j(xy)\rho(y)\mu=\rho(y)j(y)j(x)\mu.
\]
Because $\mu$ is nonzero and $\rho(y)>0$, it follows that
\[
j(xy)=j(x)j(y).
\]
\end{proof}

\begin{remark}
This is the cleanest rigorous form of the cocycle calculation.  The right
quasi-invariance hypothesis is essential for the derivation as written, because
$(N1)$ mixes left and right translations.  Without some control on
$(R_y)_*\mu$, one cannot simply cancel the right-translation term.
\end{remark}

The proposition shows that, under two-sided quasi-invariance, the left modular
cocycle behaves multiplicatively with respect to the quasigroup product.  This
is formally analogous to the multiplicativity of the modular function on a
locally compact group.

\begin{remark}
At this point one should be cautious.  The identity
\[
j(xy)=j(x)j(y)
\]
does not by itself prove that $j\equiv 1$.  Additional structural input is
required.  The main point of the present paper is that $(N1)$ naturally drives
the measure-theoretic structure toward a group-like modular behavior, and that
this behavior is plausibly related to loop structure.
\end{remark}

\begin{remark}
The multiplicativity relation for $j$ does not by itself imply
that the cocycle is trivial.  In the group case, nontrivial
characters frequently exist.  The possibility that Moufang-type
identities impose additional rigidity on such multiplicative
functions remains an interesting open question.
\end{remark}

\section{Representation of the translation group}

\begin{theorem}[Translation representation]
Let $Q$ be a quasigroup and let
\[
\chi:Q\to\mathbb{R}_{>0}
\]
be a multiplicative function satisfying
\[
\chi(xy)=\chi(x)\chi(y).
\]
Then the map
\[
\pi : \mathrm{LMlt}(Q) \to \mathbb{R}_{>0}
\]
defined by
\[
\pi(L_a)=\chi(a)
\]
extends uniquely to a group homomorphism
\[
\pi:\mathrm{LMlt}(Q)\to\mathbb{R}_{>0}.
\]
\end{theorem}

\begin{proof}

Let $\mathrm{LMlt}(Q)$ be the group generated by the left
translations $L_a$ $(a\in Q)$.

Define $\pi$ on generators by
\[
\pi(L_a)=\chi(a).
\]

Consider two generators $L_a,L_b$.  
Their composition satisfies

\[
(L_a\circ L_b)(x)=a(bx).
\]

Using multiplicativity of $\chi$, we compute

\[
\chi(a(bx))
=
\chi(a)\chi(bx)
=
\chi(a)\chi(b)\chi(x).
\]

Thus

\[
\chi((L_a\circ L_b)(x))
=
\pi(L_a)\pi(L_b)\chi(x).
\]

Therefore the assignment
\[
L_a \mapsto \chi(a)
\]
is multiplicative with respect to composition of
left translations.

Since $\mathrm{LMlt}(Q)$ is generated by the $L_a$,
this determines a unique group homomorphism

\[
\pi:\mathrm{LMlt}(Q)\to\mathbb{R}_{>0}.
\]

\end{proof}

\begin{remark}
If the modular cocycle arising from a Haar-type measure
is multiplicative, the theorem shows that it determines
a one-dimensional representation of the left multiplication
group of the quasigroup.
\end{remark}

\section{Group Case}

\begin{proposition}[Group case]
Let $G$ be a locally compact group and let $\mu$ be a left Haar measure.
Then for every $a\in G$ there exists a positive number $\Delta(a)$ such that
\[
(R_a)_*\mu = \Delta(a)\mu.
\]
The function
\[
\Delta:G\to\mathbb{R}_{>0}
\]
is a continuous group homomorphism satisfying
\[
\Delta(ab)=\Delta(a)\Delta(b).
\]
\end{proposition}

\begin{proof}
Let $\mu$ be a left Haar measure on $G$.
By left invariance we have
\[
(L_a)_*\mu=\mu
\qquad (a\in G).
\]

For right translations, the measure $(R_a)_*\mu$ is again a
left invariant measure.  Since Haar measure is unique up to
a positive scalar, there exists $\Delta(a)>0$ such that
\[
(R_a)_*\mu=\Delta(a)\mu .
\]

Let $a,b\in G$.  Using functoriality of pushforward we compute

\[
(R_{ab})_*\mu
=
(R_b\circ R_a)_*\mu
=
(R_b)_*((R_a)_*\mu).
\]

Substituting the definition of $\Delta$ gives

\[
(R_b)_*(\Delta(a)\mu)
=
\Delta(a)(R_b)_*\mu
=
\Delta(a)\Delta(b)\mu .
\]

Thus

\[
\Delta(ab)=\Delta(a)\Delta(b).
\]

Continuity of $\Delta$ follows from the continuity of
the translation action.
\end{proof}

\begin{remark}
In the group case the modular function $\Delta$ measures the
defect of right invariance of the Haar measure.
The multiplicativity of $\Delta$ is a direct consequence of
the associativity of the group operation.

In the quasigroup setting considered in this paper,
multiplicativity of the modular cocycle is instead
derived from the Moufang identity.
\end{remark}

\section{The $ax+b$ Group as a Concrete Model}

We now record a standard example in full detail.

Let
\[
G=\{(a,b)\mid a>0,\ b\in\mathbb{R}\}
\]
with multiplication
\[
(a,b)(a',b')=(aa',\,b+ab').
\]
This is the group of orientation-preserving affine transformations of the real
line,
\[
x\mapsto ax+b.
\]

\subsection{Left Haar measure}

We claim that
\[
d\mu(a,b)=\frac{da\,db}{a^2}
\]
is a left Haar measure on $G$.

Fix $(\alpha,\beta)\in G$.  Left translation is
\[
L_{(\alpha,\beta)}(a,b)=(\alpha a,\ \beta+\alpha b).
\]
Set
\[
u=\alpha a,\qquad v=\beta+\alpha b.
\]
Then
\[
a=\frac{u}{\alpha},\qquad b=\frac{v-\beta}{\alpha},
\]
and the Jacobian determinant is
\[
\left|\frac{\partial(u,v)}{\partial(a,b)}\right|=\alpha^2.
\]
Hence
\[
du\,dv=\alpha^2\,da\,db.
\]
Also,
\[
u^2=(\alpha a)^2=\alpha^2 a^2.
\]
Therefore
\[
\frac{du\,dv}{u^2}
=
\frac{\alpha^2\,da\,db}{\alpha^2 a^2}
=
\frac{da\,db}{a^2}.
\]
Thus the density $a^{-2}\,da\,db$ is invariant under left translation, proving
that it is a left Haar measure.

\subsection{Right translations and the modular function}

Now fix $(\alpha,\beta)\in G$ and consider right translation:
\[
R_{(\alpha,\beta)}(a,b)=(a\alpha,\ b+a\beta).
\]
Set
\[
u=a\alpha,\qquad v=b+a\beta.
\]
Then
\[
a=\frac{u}{\alpha},\qquad b=v-\frac{u}{\alpha}\beta.
\]
The Jacobian matrix is
\[
\begin{pmatrix}
\alpha & 0\\
\beta & 1
\end{pmatrix},
\]
so
\[
\left|\frac{\partial(u,v)}{\partial(a,b)}\right|=\alpha.
\]
Hence
\[
du\,dv=\alpha\,da\,db.
\]
Using again $u=\alpha a$, we obtain
\[
\frac{du\,dv}{u^2}
=
\frac{\alpha\,da\,db}{\alpha^2 a^2}
=
\frac{1}{\alpha}\frac{da\,db}{a^2}.
\]
Equivalently,
\[
\frac{da\,db}{a^2}
=
\alpha\frac{du\,dv}{u^2}.
\]
Thus under right translation,
\[
(R_{(\alpha,\beta)})_*\mu = \alpha\,\mu.
\]

Comparing with the standard defining relation
\[
(R_g)_*\mu=\Delta(g^{-1})\mu,
\]
we find
\[
\Delta((\alpha,\beta)^{-1})=\alpha.
\]
Since
\[
(\alpha,\beta)^{-1}=\left(\alpha^{-1},-\alpha^{-1}\beta\right),
\]
it follows that
\[
\Delta(a,b)=a^{-1}.
\]

This example is useful because it makes the modular defect entirely explicit.
In the quasigroup setting, the cocycle $j$ should be regarded as an analogue of
this phenomenon.

\section{Existence Questions and the Role of the Translation Group}

The preceding sections explain how a modular cocycle behaves once a suitable
measure is given.  The much harder question is whether such measures exist.

For groups, existence comes from the Haar theorem.  For quasigroups, the lack
of associativity obstructs a direct analogue.  A more promising approach is to
look not first at the quasigroup itself, but at the transformation groups
generated by its translations.

If $\mathrm{LMlt}(Q)$ or $\mathrm{Mlt}(Q)$ carries a locally compact group
topology acting continuously on $Q$, and if $Q$ can be modeled as a homogeneous
space for that action, then one may hope to induce a quasi-invariant measure on
$Q$ from a Haar measure on the acting group.  This line of thought suggests
that the correct general setting may involve transformation groups, groupoids,
or measure classes glued from local translation data.

We do not attempt to solve this existence problem here.  Instead, the present
paper isolates the cocycle mechanism that would become relevant once such a
measure is available.

\section{Towards a Measure-Theoretic Interpretation of Kunen's Theorem}

The purpose of this paper is to develop a framework in which
Haar-type structures on topological quasigroups can be studied
via quasi-invariant measures and modular cocycles.
From this measure-theoretic viewpoint, loop identities may be
interpreted as a form of unimodularity in the translation geometry
of quasigroups.

This perspective naturally leads to the following conjecture.

\begin{conjecture}[Modular Collapse Conjecture]
Let $Q$ be a locally compact Hausdorff topological quasigroup
satisfying the Moufang-type identity $(N1)$.
Assume that $Q$ admits a quasi-invariant measure
with modular cocycle $j$.

Then the cocycle collapses:
\[
j \equiv 1 .
\]

In particular, the translation geometry becomes
unimodular in the measure-theoretic sense.
\end{conjecture}

This conjecture is consistent with Kunen's theorem,
which shows that the identity $(N1)$ already imposes
strong structural constraints on quasigroups.

\begin{remark}
A similar conjecture should hold under a Moufang identity
in place of $(N1)$. The identity $(N1)$ is emphasized here
because of its connection with Kunen's theorem.
\end{remark}

This conjecture suggests that Kunen's theorem may admit
a measure-theoretic interpretation in terms of unimodularity
of the translation geometry.

\section{Translation--cocycle compatibility}

\begin{theorem}[Translation--cocycle compatibility]\label{thm:translation-cocycle}
Let $Q$ be a locally compact Hausdorff topological quasigroup satisfying
the Moufang identity
\[
((xy)z)y = x(y(zy)).
\]
Assume that $\mu$ is a nonzero regular Borel measure on $Q$ such that
\[
(L_a)_*\mu = j(a)\mu,
\qquad
(R_a)_*\mu = \rho(a)\mu
\]
for positive functions
\[
j,\rho:Q\to\mathbb{R}_{>0}.
\]
Then, for all $x,y\in Q$,
\[
j(xy)\rho(y)=j(x)j(y)\rho(y).
\]
Since $\rho(y)>0$, it follows that
\[
j(xy)=j(x)j(y).
\]
\end{theorem}

\begin{proof}
By Theorem~\ref{thm:translation-form-N1}, the Moufang identity $(N1)$ is
equivalent to the operator identity
\[
R_y\circ L_{xy}=L_x\circ L_y\circ R_y
\qquad (x,y\in Q).
\]
Fix $x,y\in Q$.
Applying push-forward to both sides of this identity, and using the
functoriality of push-forward, we obtain
\[
(R_y\circ L_{xy})_*\mu
=
(R_y)_*((L_{xy})_*\mu)
\]
and
\[
(L_x\circ L_y\circ R_y)_*\mu
=
(L_x)_*((L_y)_*((R_y)_*\mu)).
\]

We now compute each side separately.

\medskip

\noindent
\textbf{Left-hand side.}
Since
\[
(L_{xy})_*\mu=j(xy)\mu,
\]
we have
\[
(R_y)_*((L_{xy})_*\mu)
=
(R_y)_*(j(xy)\mu).
\]
Because $j(xy)$ is a scalar, it factors out of the push-forward:
\[
(R_y)_*(j(xy)\mu)=j(xy)(R_y)_*\mu.
\]
Using the right quasi-invariance relation
\[
(R_y)_*\mu=\rho(y)\mu,
\]
we get
\[
(R_y)_*((L_{xy})_*\mu)=j(xy)\rho(y)\mu.
\]

\medskip

\noindent
\textbf{Right-hand side.}
First use
\[
(R_y)_*\mu=\rho(y)\mu.
\]
Applying $(L_y)_*$ gives
\[
(L_y)_*((R_y)_*\mu)
=
(L_y)_*(\rho(y)\mu)
=
\rho(y)(L_y)_*\mu
=
\rho(y)j(y)\mu.
\]
Now apply $(L_x)_*$:
\[
(L_x)_*((L_y)_*((R_y)_*\mu))
=
(L_x)_*(\rho(y)j(y)\mu).
\]
Again the scalar factors out:
\[
(L_x)_*(\rho(y)j(y)\mu)
=
\rho(y)j(y)(L_x)_*\mu.
\]
Using
\[
(L_x)_*\mu=j(x)\mu,
\]
we obtain
\[
(L_x)_*((L_y)_*((R_y)_*\mu))
=
\rho(y)j(y)j(x)\mu.
\]

\medskip

Since
\[
R_y\circ L_{xy}=L_x\circ L_y\circ R_y,
\]
their push-forwards of $\mu$ must coincide.  Therefore
\[
j(xy)\rho(y)\mu
=
\rho(y)j(y)j(x)\mu.
\]
Because $\mu$ is nonzero and $\rho(y)>0$, we conclude that
\[
j(xy)=j(x)j(y).
\]
This proves the theorem.
\end{proof}

\section{Lemma for Normalization of multiplicative cocycles}

\begin{remark}
The normalization condition $\chi(e)=1$ is the quasigroup analogue
of the standard normalization for characters on groups.
In the present context it shows that once loop structure emerges,
the modular cocycles are automatically normalized at the identity.
\end{remark}

\begin{lemma}[Normalization of multiplicative cocycles]
Let $Q$ be a quasigroup and let
\[
\chi:Q\to\mathbb{R}_{>0}
\]
be a multiplicative function, i.e.
\[
\chi(xy)=\chi(x)\chi(y)
\qquad (x,y\in Q).
\]
If $Q$ has a two-sided identity element $e$, then
\[
\chi(e)=1.
\]
In particular, if a modular cocycle $j$ or $\rho$ is multiplicative on a loop,
then it is normalized at the identity.
\end{lemma}

\begin{proof}
Let $e$ be a two-sided identity element of $Q$.
Then
\[
ee=e.
\]
Applying multiplicativity of $\chi$, we obtain
\[
\chi(e)=\chi(ee)=\chi(e)\chi(e)=\chi(e)^2.
\]
Since $\chi(e)>0$, it follows that
\[
\chi(e)=1.
\]
\end{proof}

\section{Unsolved Problems}

\begin{problem}[Modular cocycle rigidity]
Let $Q$ be a locally compact topological Moufang quasigroup.
Suppose there exists a quasi-invariant measure $\mu$ such that
\[
(L_a)_*\mu = j(a)\mu .
\]
Assume that the modular cocycle $j$ satisfies
\[
j(xy)=j(x)j(y).
\]
Under what conditions must $j$ be trivial, i.e.
\[
j \equiv 1 ?
\]
\end{problem}

\begin{problem}
Let $Q$ be a locally compact topological quasigroup.
Under what conditions does there exist a quasi-invariant
measure satisfying
\[
(L_a)_*\mu = j(a)\mu ?
\]
\end{problem}

\section{Existence Issues and \\the Unimodularity Perspective}

\subsection{The existence problem}

In contrast to the classical Haar theorem for locally compact groups,
the existence of quasi-invariant measures on general locally compact
topological quasigroups is not currently known in full generality.

Accordingly, the results of this paper should be interpreted as
conditional statements describing structural consequences of the
existence of such measures.  More precisely, throughout the paper
we assume the existence of a regular Borel measure $\mu$ satisfying
the quasi-invariance relations
\[
(L_a)_*\mu = j(a)\mu,
\qquad
(R_a)_*\mu = \rho(a)\mu .
\]
Under this assumption we derived a compatibility relation between
the translation structure and the modular cocycle, and showed that
the Moufang identity forces the multiplicativity relation
\[
j(xy)=j(x)j(y).
\]

Thus the main contribution of the present work is to identify
structural consequences of the existence of quasi-invariant measures
on topological quasigroups.

\subsection{Finite quasigroups as a basic example}

Although the existence problem is open in general, finite quasigroups
provide a natural class of examples in which the framework becomes
trivial but consistent.

\begin{proposition}
Let $Q$ be a finite quasigroup and let $\mu$ be the counting measure
\[
\mu(E)=|E|.
\]
Then $\mu$ is invariant under all left and right translations.
Consequently, the associated modular cocycle is trivial,
\[
j\equiv 1, \qquad \rho\equiv 1 .
\]
\end{proposition}

\begin{proof}
Since every translation $L_a$ and $R_a$ is a bijection of the finite
set $Q$, the counting measure is preserved:
\[
(L_a)_*\mu=\mu, \qquad (R_a)_*\mu=\mu .
\]
Thus the cocycle factors are equal to $1$.
\end{proof}

This simple example shows that the modular cocycle may be interpreted
as measuring the failure of translation invariance beyond the finite
case.

\subsection{Unimodularity and loop structure}

In the theory of locally compact groups, the modular function
\[
\Delta:G\to\mathbb{R}_{>0}
\]
measures the defect of right invariance of Haar measure.
A group is called \emph{unimodular} if $\Delta\equiv 1$.

The present framework suggests a similar interpretation in the
quasigroup setting.  The function $j:Q\to\mathbb{R}_{>0}$ obtained
from quasi-invariance plays the role of a modular cocycle.

\section{Conclusion}

We introduced a framework for studying Haar-type structures on
topological quasigroups using quasi-invariant measures and modular
cocycles. Under explicit two-sided quasi-invariance assumptions,
we showed that the Moufang-type identity $(N1)$ yields a rigorous
multiplicativity relation for the left cocycle.

This places the quasigroup situation in close analogy with the
classical theory of the modular function on locally compact groups.

Several problems remain open, most importantly the existence theory
for quasi-invariant measures on topological quasigroups and the
precise mechanism by which cocycle rigidity might force the existence
of an identity element. These questions suggest that the interaction
among quasigroup theory, topology, and harmonic analysis deserves
further study.

\begin{remark}
The purpose of this note is not to provide a new proof of
Kunen's theorem. Rather, we propose a measure-theoretic
framework in which the Moufang identity interacts with
translation geometry through modular cocycles.
We hope that this viewpoint may stimulate further work
connecting quasigroup theory with harmonic analysis.
\end{remark}

\begin{remark}
In this paper we proposed a framework for studying Haar-type
structures on topological quasigroups using quasi-invariant
measures and modular cocycles. This perspective suggests that
loop identities may admit a measure-theoretic interpretation
in terms of unimodularity of the translation geometry.
In particular, the conjectural collapse of the modular cocycle
under the identity $(N1)$ would provide a new viewpoint on
Kunen's theorem from the perspective of harmonic analysis on
quasigroup-like structures.
\end{remark}

\paragraph{Future directions.}
Several natural questions arise from the present framework.
First, it would be important to develop a general existence theory
for quasi-invariant measures on locally compact topological quasigroups.
Second, one may ask under what structural conditions the rigidity of
the modular cocycle forces the existence of an identity element,
thereby producing loop structure.
These problems suggest that the interaction between quasigroup theory,
topology, and harmonic analysis deserves further investigation.

\appendix

\section{Structural Diagram of the Measure-Theoretic Framework}

For the convenience of the reader, we summarize the conceptual
architecture of the paper.

\paragraph{Interpretation.}
The diagram clarifies the conceptual mechanism developed in this paper.
The Moufang-type identity $(N1)$ constrains the translation calculus
of the quasigroup.  When a quasi-invariant measure exists, this
constraint propagates to the modular cocycle and forces a rigidity
relation.  The resulting multiplicativity of the cocycle suggests
a representation-theoretic interpretation of the translation geometry.
In this sense, the emergence of loop structure in Kunen's theorem
may be interpreted as a collapse of the modular defect.

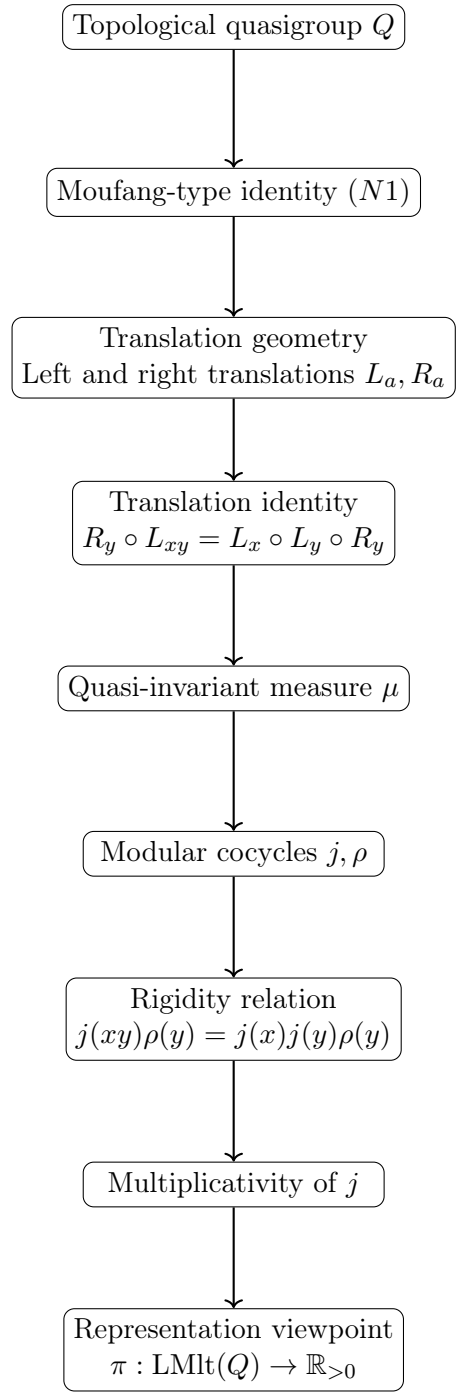
\begin{center}
\begin{tikzpicture}[
node distance=2.2cm,
box/.style={rectangle, draw, rounded corners, align=center, minimum width=4cm},
arrow/.style={->, thick}
]

\node[box] (Q) {Topological quasigroup $Q$};

\node[box, below of=Q] (N1) {Moufang-type identity $(N1)$};

\node[box, below of=N1] (T) {Translation geometry\\
Left and right translations $L_a, R_a$};

\node[box, below of=T] (TI) {Translation identity\\
$R_y \circ L_{xy} = L_x \circ L_y \circ R_y$};

\node[box, below of=TI] (M) {Quasi-invariant measure $\mu$};

\node[box, below of=M] (C) {Modular cocycles $j,\rho$};

\node[box, below of=C] (R) {Rigidity relation\\
$j(xy)\rho(y)=j(x)j(y)\rho(y)$};

\node[box, below of=R] (Mul) {Multiplicativity of $j$};

\node[box, below of=Mul] (Rep) {Representation viewpoint\\
$\pi:\mathrm{LMlt}(Q)\to\mathbb{R}_{>0}$};

\draw[arrow] (Q) -- (N1);
\draw[arrow] (N1) -- (T);
\draw[arrow] (T) -- (TI);
\draw[arrow] (TI) -- (M);
\draw[arrow] (M) -- (C);
\draw[arrow] (C) -- (R);
\draw[arrow] (R) -- (Mul);
\draw[arrow] (Mul) -- (Rep);

\end{tikzpicture}
\end{center}

\paragraph{Kunen collapse mechanism.}
Kunen's theorem shows that the Moufang-type identity $(N1)$
imposes strong constraints on the translation structure of a quasigroup.
These constraints force the existence of a two-sided identity element,
and hence the quasigroup collapses to a loop.
The diagram illustrates this structural mechanism.

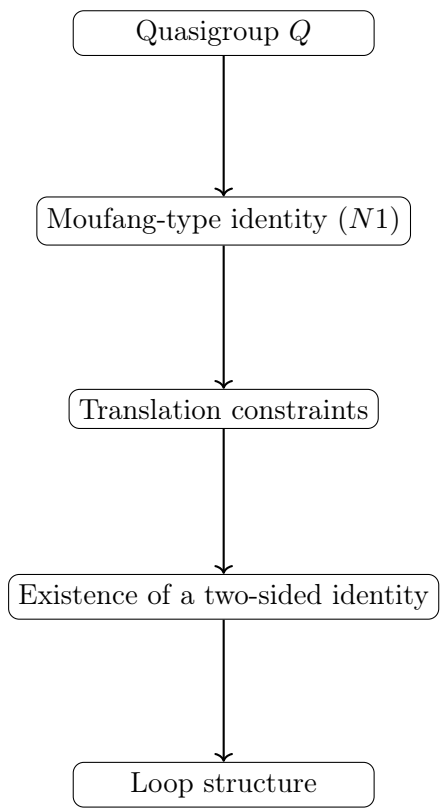
\begin{center}
\begin{tikzpicture}[
node distance=2.5cm,
box/.style={rectangle, draw, rounded corners, align=center, minimum width=4cm},
arrow/.style={->, thick}
]

\node[box] (Q) {Quasigroup $Q$};

\node[box, below of=Q] (N1) {Moufang-type identity $(N1)$};

\node[box, below of=N1] (TR) {Translation constraints};

\node[box, below of=TR] (FIX) {Existence of a two-sided identity};

\node[box, below of=FIX] (LOOP) {Loop structure};

\draw[arrow] (Q) -- (N1);
\draw[arrow] (N1) -- (TR);
\draw[arrow] (TR) -- (FIX);
\draw[arrow] (FIX) -- (LOOP);

\end{tikzpicture}
\end{center}

\noindent Takao Inou\'{e}

\noindent Faculty of Informatics

\noindent Yamato University

\noindent Katayama-cho 2-5-1, Suita, Osaka, 564-0082, Japan

\noindent inoue.takao@yamato-u.ac.jp
 
\noindent (Personal) takaoapple@gmail.com (I prefer my personal mail)

\end{document}